\documentclass[11pt]{article}

\date{}
\title{Letter graphs and geometric grid classes of permutations:  characterization and recognition
\footnote{An extended abstract of this paper was published in \cite{iwoca}.}} 

\author{Bogdan Alecu\thanks{Mathematics Institute, University of Warwick, Coventry CV4 7AL, UK.} 
\and Vadim Lozin\thanks{Mathematics Institute, University of Warwick, Coventry CV4 7AL, UK.} 
\and Dominique de Werra\thanks{EPFL, Institute of Mathematics, CH-1015 Lausanne, Switzerland.}
\and Viktor Zamaraev\thanks{Department of Computer Science, Durham University, South Road, Durham, DH1 3LE, UK.}}

\usepackage{graphicx}
\usepackage{tabularx}
\usepackage{url}
\usepackage{amsthm} 
\usepackage{amsmath} 
\usepackage{amsfonts}
\usepackage{amssymb}
\usepackage{tikz}
\usetikzlibrary{decorations.pathreplacing}
\tikzstyle{vertex}=[circle,fill=black!100,text=white,inner sep=0.8mm]
\tikzstyle{point}=[circle,fill=black,inner sep=0.1mm]

\begin{document}
\maketitle

\newtheorem{proposition}{Proposition}
\newtheorem*{example}{Example}
\newtheorem{theorem}{Theorem}
\newtheorem{lemma}{Lemma}
\newtheorem{cor}{Corollary}
\theoremstyle{definition}
\newtheorem{definition}{Definition}
\newtheorem{remark}{Remark}
\newtheorem{conjecture}{Conjecture}

\begin{abstract}
In this paper, we reveal an intriguing relationship between two seemingly unrelated notions: letter graphs and geometric grid classes of permutations. 
An important property common for both of them is well-quasi-orderability, implying, in a non-constructive way, a polynomial-time recognition of 
geometric grid classes of permutations  and $k$-letter graphs for a fixed $k$. However, constructive algorithms are available only for $k=2$.
In this paper, we  present the first constructive polynomial-time algorithm for the recognition of $3$-letter graphs. It is based on a structural 
characterization of graphs in this class. 
\end{abstract}

\section{Introduction}

Letter graphs and geometric grid classes of permutations have been introduced independently of each other  in \cite{letter-graphs}  and \cite{grid-classes}, respectively. 
Nothing in the definitions of these notions suggests any connection between them.
We believe that  letter graphs and geometric grid classes of permutations can be connected through the notion of permutation graph. Speaking informally,
we believe that geometric grid classes of permutations and letter graphs are two languages describing the same concept in the universe of permutations and  
permutation graphs, respectively. We state this formally as a conjecture as follows:
\begin{conjecture}\label{con:main} 
Let $X$ be a class of permutations and ${\cal G}_X$ the corresponding class of permutation graphs.
Then $X$ is geometrically griddable if and only if ${\cal G}_X$ is a class of $k$-letter graphs for a finite value of $k$.
\end{conjecture}
In this conjecture, the parameter $k$ stands for the size of the alphabet used to describe graphs by means of letters (all definitions will be given in Section~\ref{sec:pre}).
In Section~\ref{sec:main}, we verify this conjecture in two cases: first, in Section~\ref{sec:from}, we prove the ``only if'' part of the conjecture, 
i.e. we translate the concept of geometric grid classes of permutations to the language of letter graphs, and then in Section~\ref{sec:back} we prove the conjecture in the reverse direction for $k=2$.

\medskip
An important property common for both of these notions is well-quasi-orderability.
It implies, in particular, that geometric grid classes of permutations and $k$-letter graphs (for a fixed $k$) can 
be described by finitely many forbidden obstructions. This proves, in a non-constructive way, that geometric grid classes of permutations  and $k$-letter graphs (for a fixed $k$)
can be recognized in polynomial time. However, constructive algorithms are not available for the recognition problem, except 
for the $2$-letter graphs and corresponding classes of permutations. As a step towards solving this problem for larger values of $k$, in Section~\ref{sec:3letter} 
we study the class of $3$-letter graphs. We provide a structural characterization of graphs in this class, which leads to 
a polynomial-time algorithm to recognize them. 
    
We finish the paper in Section~\ref{sec:conclusion} by positioning graph lettericity in the hierarchy of other graph parameters and discussing open problems in the area.     
All preliminary information related to the topic of the paper can be found in Section~\ref{sec:pre}.    


\section{Preliminaries}
\label{sec:pre}


All graphs in this paper are finite, undirected, without loops and multiple edges. The vertex set and the edge set of a graph $G$
are denoted by $V(G)$ and $E(G)$, respectively. For a vertex $x\in V(G)$ we denote by $N(x)$ the neighbourhood of $x$, i.e. the set of 
vertices of $G$ adjacent to $x$. A subgraph of $G$ induced by a subset of vertices $U\subseteq V(G)$ is denoted $G[U]$.
By $\overline{G}$ we denote the complement of $G$.

A {\it clique} in a graph is a subset of pairwise adjacent vertices and an {\it independent set} is a subset of pairwise non-adjacent vertices. 
A graph $G$ is {\it bipartite} if $V(G)$ can be partitioned into two independent sets and $G$ is {\it split} if  $V(G)$ can be partitioned into
an independent set and a clique. 

By $K_n$, $C_n$ and $P_n$ we denote the complete graph, the chordless cycle and the chordless path with $n$ vertices, respectively.
Also, $G+H$ denotes the disjoint union of two graphs $G$ and $H$. In particular, $pG$ is the disjoint union of $p$ copies of $G$. 

In the rest of this section we introduce some classes of graphs and permutations relevant to the topic of the paper. 

\subsection{Chain graphs} 

A graph $G$ is a chain graph if it is bipartite and admits a bipartition $V(G)=V_1\cup V_2$ such that for any 
two vertices $x,y$ in the same part $V_{i}$ either $N(x)\subseteq N(y)$ or $N(y)\subseteq N(x)$. In other words, the vertices in each part 
of the bipartition of $G$ can be linearly ordered under inclusion of their neighbourhoods, i.e. they form a chain. In terms of minimal forbidden induced subgraphs,
the chain graphs are precisely the $2K_2$-free bipartite graphs. Figure~\ref{fig:Z5} represents an example of a chain graph $Z_5$ containing 5 vertices in 
each part. In an obvious way this example can be extended to $Z_n$ for any value of $n$. The importance of this graph is due to its universality:
$Z_n$ contains all $n$-vertex chain graphs as induced subgraphs \cite{chain-uni}. 

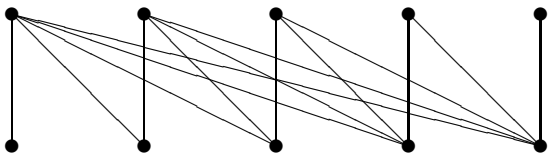
\begin{figure}[ht]
\begin{center}
\begin{picture}(300,60)
\put(50,0){\circle*{5}}
\put(100,0){\circle*{5}}
\put(150,0){\circle*{5}}
\put(200,0){\circle*{5}}
\put(250,0){\circle*{5}}



\put(50,50){\circle*{5}}
\put(100,50){\circle*{5}}
\put(150,50){\circle*{5}}
\put(200,50){\circle*{5}}
\put(250,50){\circle*{5}}

\put(50,0){\line(0,1){50}}
\put(250,0){\line(-1,1){50}}
\put(250,0){\line(-2,1){100}}
\put(250,0){\line(-3,1){150}}
\put(250,0){\line(-4,1){200}}
\put(100,0){\line(0,1){50}}
\put(200,0){\line(-1,1){50}}
\put(200,0){\line(-2,1){100}}
\put(200,0){\line(-3,1){150}}
\put(150,0){\line(0,1){50}}
\put(150,0){\line(-1,1){50}}
\put(150,0){\line(-2,1){100}}
\put(200,0){\line(0,1){50}}
\put(100,0){\line(-1,1){50}}
\put(250,0){\line(0,1){50}}

\end{picture}
\end{center}
\caption{The graph $Z_5$}
\label{fig:Z5}
\end{figure}

\subsection{Threshold graphs} 

The class of threshold graphs was introduced in \cite{threshold-graphs}, where it was characterized in terms of minimal forbidden induced subgraphs as follows:
a graph $G$ is threshold if and only if it is $(P_4,C_4,2K_2)$-free. This class is closely related to the class 
of chain graphs in the sense that if we create a clique in one of the parts of a chain graph, then the graph transforms  
into a threshold graph and vice versa. Moreover, by transforming $Z_n$ in this way we obtain 
an $n$-universal threshold graph, i.e. a threshold graph containing all $n$-vertex threshold graphs as 
induced subgraphs, see e.g. \cite{universal-threshold}. More about threshold graphs can be found in the book \cite{threshold}
devoted to this class.

\subsection{Permutation graphs} 

Let $\pi$ be a permutation of the set $\{1,2,\ldots,n\}$. The permutation graph $G_{\pi}$ of this permutation has
$\{1,2,\ldots,n\}$ as its vertex set with $i$ and $j$ being adjacent if and only if $(i-j)(\pi(i)-\pi(j))<0$. 
A graph $G$ is a {\it permutation graph} if there is a permutation $\pi$ such that $G$ is isomorphic to $G_{\pi}$. 

Alternatively, a permutation graph can be defined as the intersection graph of line segments between two parallel lines: 
each segment represents a vertex and two vertices are adjacent if and only if the corresponding segments cross each other.
For instance, Figure~\ref{fig:pi} represents the permutation 415263 written at the bottom of the diagram. It is not difficult 
to see that the permutation graph of this permutation is the chain graph $Z_3$. Clearly, this example can be extended 
to a diagram representing $Z_n$ for any value of $n$. Therefore, all chain graphs are permutation graphs. It is also known
that all threshold (and more generally, all $P_4$-free) graphs are permutation graphs. It is an interesting exercise 
(left to the reader) to construct a (diagram of) permutation representing the universal threshold graph.

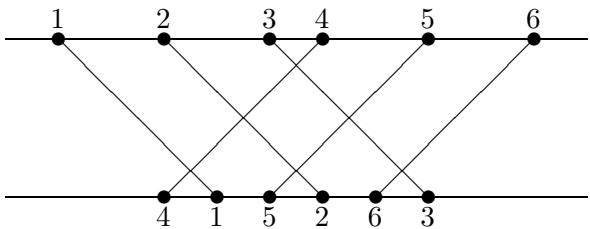
\begin{figure}[ht]
\begin{center}
\begin{picture}(220,90)
\put(40,0){\circle*{5}} 
\put(60,0){\circle*{5}}
\put(80,0){\circle*{5}} 
\put(100,0){\circle*{5}}
\put(120,0){\circle*{5}} 
\put(140,0){\circle*{5}}

\put(0,60){\circle*{5}} 
\put(40,60){\circle*{5}} 
\put(80,60){\circle*{5}} 
\put(100,60){\circle*{5}}
\put(140,60){\circle*{5}}
\put(180,60){\circle*{5}}

\put(-20,0){\line(1,0){220}} 
\put(-20,60){\line(1,0){220}} 
\put(40,0){\line(1,1){60}} 
\put(60,0){\line(-1,1){60}}
\put(80,0){\line(1,1){60}} 
\put(100,0){\line(-1,1){60}}
\put(120,0){\line(1,1){60}} 
\put(140,0){\line(-1,1){60}}
\put(-3,64){$1$}
\put(37,-11){$4$} 
\put(57,-11){$1$}
\put(37,64){$2$}
\put(77,-11){$5$} 
\put(97,-11){$2$}
\put(77,64){$3$}
\put(97,64){$4$}

\put(117,-11){$6$} 
\put(137,-11){$3$}
\put(137,64){$5$}
\put(177,64){$6$}
\end{picture}
\end{center}
\caption{A diagram representing the permutation 415263} 
\label{fig:pi}
\end{figure}


\subsection{Letter graphs}
\label{sec:letter-graphs}

 
Let $\Sigma$ be a finite alphabet and ${\cal P}\subseteq \Sigma^2$ a set of ordered pairs of symbols from $\Sigma$, called the {\it decoder}.
To each word $w=w_1w_2\cdots w_n$ with $w_i\in \Sigma$ we associate a graph $G({\cal P},w)$, called the {\it letter graph}
of $w$, by defining $V(G({\cal P},w))=\{1,2,\ldots ,n\}$ with $i$ being adjacent to $j>i$ if and only if 
the ordered pair $(w_i,w_j)$ belongs to the decoder $\cal P$.

It is not difficult to see that every graph $G$ is a letter graph in an  alphabet of size at most $|V(G)|$ with an appropriate decoder $\cal P$.
The minimum $\ell$ such that $G$ is a letter graph in an alphabet of $\ell$ letters is the {\it lettericity} of $G$ and is denoted
$\ell(G)$. A graph is a $k$-letter graph if its lettericity is at most $k$.

The notion of $k$-letter graphs was introduced in \cite{letter-graphs} and in the same paper the author characterized 
$k$-letter graphs as follows.

\begin{theorem}\label{thm:k-letter} A graph $G$ is a $k$-letter graph if and only if 
\begin{itemize}
\item[$1.$] there is a partition $V_1,V_2,\ldots,V_p$ of $V(G)$ with $p\le k$ such that each $V_i$ is either a clique or an independent set
in $G$, and
\item[$2.$] there is a linear ordering $L$ of $V(G)$ such that for each pair of distinct indices $1\le i,j\le p$, the intersection of 
$E(G)$ with $V_i\times V_j$ is one of the following four types (where $L$ is considered as a binary relation, i.e. as a set of pairs):
\begin{itemize}
\item[{\rm (}a{\rm )}] $L\cap (V_i\times V_j)$;
\item[{\rm (}b{\rm )}] $L^{-1}\cap (V_i\times V_j)$;
\item[{\rm (}c{\rm )}] $V_i\times V_j$;
\item[{\rm (}d{\rm )}] $\emptyset$.
\end{itemize}  
\end{itemize}
\end{theorem}

\medskip
\noindent
{\it Example}.  Consider the alphabet $\Sigma=\{a,b\}$ and decoder ${\cal P}=\{(a,b)\}$. It is not difficult to see that 
for any word $w$ the graph $G({\cal P},w)$ is a chain graph. Indeed, in this graph the $a$ vertices (i.e. the vertices labelled by $a$)
form an independent set and the $b$ vertices form an independent set. Besides, each of these two sets forms a chain defined by the order in
which the vertices appear in the word. Also, it is not difficult to see that the periodic word $abab\ldots abab$ of length $2n$ defines 
the chain graph $Z_n$. This observation provides an alternative proof of the universality of $Z_n$, since every 
word of length $n$ is a subword of $abab\ldots abab$. 
If we add to the decoder the pair $(a,a)$, the graph $Z_n$ transforms into the $n$-universal threshold graph. 
Therefore, all chain graphs and all threshold graphs have lettericity at most 2.

\medskip
The notion of letter graphs is of interest for various reasons. First, some important graph classes, such as chain graphs or threshold graphs,
can be described in the terminology of letter graphs.  
Second, letter graphs provide an interesting contribution to the theory of ordered graphs, i.e. graphs given together with 
a linear order of its vertices; for more information on this notion see e.g. \cite{ordered}.
Third, graph lettericity contributes to the rich theory of graph parameters. We discuss this topic in Section~\ref{sec:conclusion}. 
Finally, and perhaps most importantly, graphs of bounded lettericity are well-quasi-ordered by the induced subgraph relation \cite{letter-graphs}. 
This is a rare property of graphs, which was shown, up to date, only for  some restricted graph classes, see e.g. \cite{Damaschke,two}. 

\subsection{Geometric grid classes of permutations}

The notion of geometric grid classes of permutations was introduced in \cite{grid-classes} as follows. Suppose that $M$ is a $0/\pm 1$ matrix. 
The {\it standard figure} of $M$ is the set of points in $\mathbb{R}^2$ consisting of 
\begin{itemize}
\item the increasing open line segment from $(k-1,\ell -1)$ to $(k,\ell)$ if $M_{k,\ell}=1$ or
\item the decreasing open line segment from $(k-1,\ell )$ to $(k,\ell-1)$ if $M_{k,\ell}=-1$.
\end{itemize}

We index matrices first by column, counting left to right, and then by row, counting bottom to top. 
The {\it geometric grid class} of $M$, denoted by ${\rm Geom}(M)$, is then the set of all permutations that can be drawn
on this figure in the following manner. Choose $n$ points in the figure, no two on a common
horizontal or vertical line. Then label the points from $1$ to $n$ from bottom to top and record
these labels reading left to right. Figure~\ref{fig:grids} represents two permutations that lie, respectively, 
in grid classes of 
$$\left(
\begin{array}{rr}
  1 & -1 \\
  -1 & 1 
\end{array}
\right)
\mbox{ and } 
\left(
\begin{array}{rr}
   -1 & 1 \\
   1 & -1 
 \end{array}\right).
$$

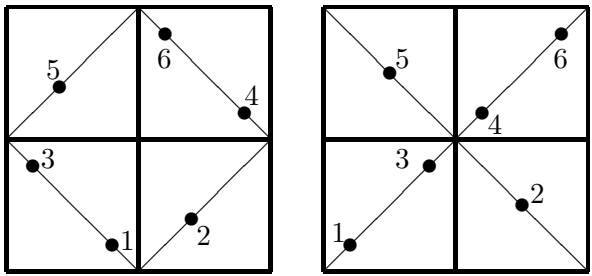
\begin{figure}[ht]
\begin{center}
\begin{picture}(220,100)

\put(0,50){\line(1,1){50}} 
\put(0,50){\line(1,-1){50}}
\put(100,50){\line(-1,1){50}} 
\put(100,50){\line(-1,-1){50}}

\put(120,0){\line(1,1){100}} 
\put(220,0){\line(-1,1){100}}

\put(40,10){\circle*{5}} \put(43,8){$1$}
\put(70,20){\circle*{5}} \put(72,10){$2$}
\put(10,40){\circle*{5}} \put(13,39){$3$}
\put(90,60){\circle*{5}} \put(90,63){$4$}
\put(20,70){\circle*{5}} \put(15,73){$5$}
\put(60,90){\circle*{5}} \put(57,77){$6$}

\put(130,10){\circle*{5}} \put(123,11){$1$}
\put(195,25){\circle*{5}} \put(198,26){$2$}
\put(160,40){\circle*{5}} \put(147,39){$3$}
\put(180,60){\circle*{5}} \put(182,52){$4$}
\put(145,75){\circle*{5}} \put(147,77){$5$}
\put(210,90){\circle*{5}} \put(207,77){$6$}

\linethickness{0.5mm}
\put(0,0){\line(1,0){100}} 
\put(0,0){\line(0,1){100}}
\put(100,100){\line(0,-1){100}} 
\put(100,100){\line(-1,0){100}}

\put(120,0){\line(0,1){100}} 
\put(120,0){\line(1,0){100}}
\put(220,100){\line(0,-1){100}} 
\put(220,100){\line(-1,0){100}}

\put(0,50){\line(1,0){100}} 
\put(50,0){\line(0,1){100}}
\put(120,50){\line(1,0){100}} 
\put(170,0){\line(0,1){100}}

\end{picture}
\caption{The permutation $351624$ on the left and the permutation $153426$ on the right.}
\label{fig:grids}
\end{center}
\end{figure}

We will say that a permutation class is {\it geometrically griddable} if it is contained in the union of finitely many geometric
grid classes. The geometrically griddable classes of permutations enjoy many nice properties. In particular, 
in~\cite{grid-classes} the following result has been proved.

\begin{theorem}\label{thm:grid}
Every geometrically griddable class of permutations is well-quasi-ordered and is in bijection with a regular language.
\end{theorem}

To define the pattern containment relation, we observe that the intersection diagram representing  a permutation (see e.g. Figure~\ref{fig:pi})
uniquely defines the permutation without the labels attached to the segments. Then a permutation $\pi$ is said to contain a permutation $\rho$ if
the intersection diagram representing $\rho$ can obtained from the  diagram representing $\pi$ by deleting some segments. 


\section{Letter graphs and geometric grid classes of permutations}
\label{sec:main}


In this section we verify Conjecture~\ref{con:main} in two cases: 
the ``only if'' direction'' (Section~\ref{sec:from}) and the case $k=2$ of the ``if'' direction'' 
(Section~\ref{sec:back}).


\subsection{From geometric grid classes of permutations to letter graphs}
\label{sec:from}


The goal of this section is to prove the following result. 

\begin{theorem}\label{thm-8-3-8} 
Let $X$ be a class of permutations and ${\cal G}_X$ the corresponding class of permutation graphs.
If $X$ is a geometric grid class, then ${\cal G}_X$ is a class of $k$-letter graphs for a finite value of $k$.
\end{theorem}

To prove Theorem~\ref{thm-8-3-8}, we first outline the correspondence (bijection)  between a geometrically griddable class of permutations and a regular language
established in Theorem~\ref{thm:grid}. To this end, we need the following definition from \cite{grid-classes}.

\begin{definition}
We say that a $0/\pm 1$ matrix $M$ of size $t\times u$ is a {\em partial multiplication matrix}
if there are column and row signs
$$
c_1, \ldots , c_t, r_1, \ldots , r_u \in \{1,-1\}
$$
such that every entry $M_{k,\ell}$ is equal to either $0$ or the product $c_kr_{\ell}$.
\end{definition}

\begin{example} The matrix\index{matrix} $\left(\begin{array}{rrr}  1 & 0 & -1 \\   -1& 1 & 0 \end{array} \right)$  
is a partial\index{partial multiplication matrix} multiplication matrix. This matrix\index{matrix} has column and row signs $c_2=c_3=r_1=1$ and $c_1=r_2=-1$.  
\end{example}

The importance of this notion for the study of geometric grid classes of permutations is due to the following 
proposition proved in~\cite{grid-classes}.

\begin{proposition}
Every geometric grid class is the geometric grid class of a partial multiplication matrix.
\end{proposition}

Let $M$ be a $t\times u$  partial multiplication matrix with column and row signs 
$$c_1, \ldots , c_t, r_1, \ldots , r_u \in \{1,-1\}$$
and let $\Phi_M$ be the standard gridded figure of $M$. We will interpret the signs of the columns and rows of $M$ as
the ``directions'' associated with the columns and rows of $\Phi_M$ with the following convention: $c_i=1$ corresponds to $\rightarrow$,
$c_i=-1$ corresponds to $\leftarrow$, $r_i=1$ corresponds to $\uparrow$, and $r_i=-1$ corresponds to $\downarrow$.  
The standard gridded figure of the matrix $M=\left(\begin{array}{rrr}  0 & 1 & 1 \\   1& -1 & -1 \end{array} \right)$ 
with row signs $r_1 = -1$ and $r_2 = 1$ and column signs $c_1 = -1$, $c_2 = c_3 = 1$ is represented in Figure~\ref{fig:partial}.

\begin{figure}[ht]
\begin{center}
\begin{picture}(120,100)

\put(50,50){\line(1,1){50}} 
\put(50,50){\line(1,-1){50}}
\put(100,50){\line(1,1){50}} 
\put(100,50){\line(1,-1){50}}
\put(50,50){\line(-1,-1){50}}

\put(-5,52){\vector(0,1){46}} 
\put(-5,48){\vector(0,-1){46}}
\put(48,-5){\vector(-1,0){46}} 
\put(52,-5){\vector(1,0){46}}
\put(102,-5){\vector(1,0){46}}

\put(107,43){\circle*{5}} \put(114,43){$p_1$}
\put(114,36){\circle*{5}} \put(110,24){$p_2$}
\put(71,71){\circle*{5}} \put(67,61){$p_3$}
\put(78,22){\circle*{5}} \put(76,28){$p_4$}
\put(15,15){\circle*{5}} \put(13,5){$p_5$}
\put(141,91){\circle*{5}} \put(137,82){$p_6$}
\put(96,96){\circle*{5}} \put(88,85){$p_7$}

\linethickness{0.5mm}
\put(0,0){\line(1,0){150}} 
\put(0,0){\line(0,1){100}}
\put(0,100){\line(1,0){150}} 



\put(0,50){\line(1,0){150}} 
\put(50,0){\line(0,1){100}}
\put(100,0){\line(0,1){100}} 
\put(150,0){\line(0,1){100}}

\end{picture}
\caption{A standard gridded figure of a partial multiplication matrix.}
\label{fig:partial}
\end{center}
\end{figure}
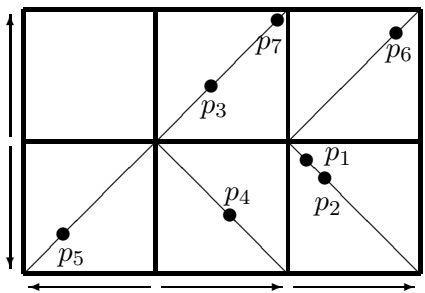

The {\it base point} of a cell $C_{k,\ell}$ of the figure $\Phi_M$ is one of the four corners of the cell, where both directions 
(associated with column $k$ and row $\ell)$ start. For instance, in Figure~\ref{fig:partial} the base point of the cell $C_{3,1}$
is the top-left corner.  

In order to establish a bijection between ${\rm Geom}(M)$ and a regular language, 
we first fix an alphabet  $\Sigma$ (known as the {\it cell alphabet} of $M$) as follows:
$$
\Sigma=\{a_{k\ell}\ :\ M_{k,\ell}\ne 0\}.
$$
Now, let $\pi$ be a permutation in ${\rm Geom}(M)$, i.e.\ a permutation represented by a set of $n$ points in the figure $\Phi_M$.
For each point $p_i$ of $\pi$, let $d_i$ be the distance from the base point of the cell containing $p_i$ to $p_i$.
Without loss of generality, we assume that these distances are pairwise different and the points are ordered so that 
$0<d_1<d_2<\cdots<d_n<1$. If $p_i$ belongs to the cell $C_{k,\ell}$ of  $\Phi_M$, we define $\phi(p_i)=a_{k\ell}$.
Then $\phi(\pi)=\phi(p_1)\phi(p_2)\cdots\phi(p_n)$ is a word in the alphabet $\Sigma$, i.e. $\phi$ defines a mapping from 
${\rm Geom}(M)$ to $\Sigma^*$. Figure~\ref{fig:partial} shows seven points defining the permutation $1527436$.
The mapping $\phi$ associates with this permutation a word in the alphabet $\Sigma$ as follows: $\phi(1527436)=a_{31}a_{31}a_{22}a_{21}a_{11}a_{32}a_{22}$.

Conversely, let $w=w_1\cdots w_n$ be a word in $\Sigma^*$ and let $0<d_1<\cdots<d_n<1$ be $n$ distances 
chosen arbitrarily. If $w_i=a_{k\ell}$, we let $p_i$ be the point on the line segment in cell $C_{k,\ell}$ at distance 
$d_i$ from the base point of $C_{k,\ell}$. The $n$ points of $\Phi_M$ constructed in this way define a permutation $\psi(w)$ 
in ${\rm Geom}(M)$. Therefore, $\psi$ is a mapping from  $\Sigma^*$ to ${\rm Geom}(M)$. 

This correspondence between $\Sigma^*$ and ${\rm Geom}(M)$ is not yet a bijection, as illustrated in Figure~\ref{fig:two},
because the order in which the points are consecutively inserted into {\it independent} cells (i.e.\ cells which share neither 
a column nor a row) is irrelevant. To turn this correspondence into a bijection, we say that two words $v,w\in \Sigma^*$ are 
{\it equivalent} if one can be obtained from the other by successively interchanging adjacent letters which represent independent cells.
The equivalence classes of this relation form a {\it trace monoid} and each element of this monoid is called a {\it trace}. 
It is known that in any trace monoid it is possible to choose a unique representative from each trace in such a way that 
the resulting set of representatives forms a regular language (see e.g. \cite{Diekert}, Corollary 1.2.3). 
This is the language which is in a bijection with ${\rm Geom}(M)$, as was shown in \cite{grid-classes}.

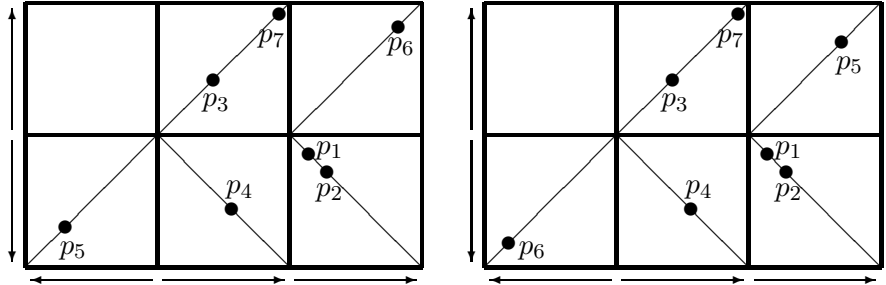
\begin{figure}[ht]
\begin{center}
\begin{picture}(170,100)

\put(50,50){\line(1,1){50}} 
\put(50,50){\line(1,-1){50}}
\put(100,50){\line(1,1){50}} 
\put(100,50){\line(1,-1){50}}
\put(50,50){\line(-1,-1){50}}

\put(-5,52){\vector(0,1){46}} 
\put(-5,48){\vector(0,-1){46}}
\put(48,-5){\vector(-1,0){46}} 
\put(52,-5){\vector(1,0){46}}
\put(102,-5){\vector(1,0){46}}

\put(107,43){\circle*{5}} \put(110,43){$p_1$}
\put(114,36){\circle*{5}} \put(110,27){$p_2$}
\put(71,71){\circle*{5}} \put(67,61){$p_3$}
\put(78,22){\circle*{5}} \put(76,28){$p_4$}
\put(15,15){\circle*{5}} \put(13,5){$p_5$}
\put(141,91){\circle*{5}} \put(137,82){$p_6$}
\put(96,96){\circle*{5}} \put(88,85){$p_7$}

\linethickness{0.5mm}
\put(0,0){\line(1,0){150}} 
\put(0,0){\line(0,1){100}}
\put(0,100){\line(1,0){150}} 



\put(0,50){\line(1,0){150}} 
\put(50,0){\line(0,1){100}}
\put(100,0){\line(0,1){100}} 
\put(150,0){\line(0,1){100}}

\end{picture}
\begin{picture}(150,100)

\put(50,50){\line(1,1){50}} 
\put(50,50){\line(1,-1){50}}
\put(100,50){\line(1,1){50}} 
\put(100,50){\line(1,-1){50}}
\put(50,50){\line(-1,-1){50}}

\put(-5,52){\vector(0,1){46}} 
\put(-5,48){\vector(0,-1){46}}
\put(48,-5){\vector(-1,0){46}} 
\put(52,-5){\vector(1,0){46}}
\put(102,-5){\vector(1,0){46}}

\put(107,43){\circle*{5}} \put(110,43){$p_1$}
\put(114,36){\circle*{5}} \put(110,27){$p_2$}
\put(71,71){\circle*{5}} \put(67,61){$p_3$}
\put(78,22){\circle*{5}} \put(76,28){$p_4$}
\put(9,9){\circle*{5}} \put(13,5){$p_6$}
\put(135,85){\circle*{5}} \put(133,75){$p_5$}
\put(96,96){\circle*{5}} \put(88,85){$p_7$}

\linethickness{0.5mm}
\put(0,0){\line(1,0){150}} 
\put(0,0){\line(0,1){100}}
\put(0,100){\line(1,0){150}} 



\put(0,50){\line(1,0){150}} 
\put(50,0){\line(0,1){100}}
\put(100,0){\line(0,1){100}} 
\put(150,0){\line(0,1){100}}

\end{picture}
\caption{Two drawings of the permutation $1527436$. The drawing on the left is encoded as $a_{31}a_{31}a_{22}a_{21}a_{11}a_{32}a_{22}$
and the drawing on the right is encoded as $a_{31}a_{31}a_{22}a_{21}a_{32}a_{11}a_{22}$.}
\label{fig:two}
\end{center}
\end{figure}

Next,  we will show that the permutation graph $G_{\pi}$ of $\pi\in {\rm Geom}(M)$ is a $k$-letter graph with $k=|\Sigma|$.
Indeed, the non-empty cells of the figure $\Phi_M$ defines a partition of the vertex set of $G_{\pi}$ into 
cliques and independent sets and the word $\phi(\pi)$ defines the order of the vertex set of $G_{\pi}$
satisfying conditions of Theorem~\ref{thm:k-letter}. More formally, let us show that the matrix $M$ uniquely defines 
a decoder ${\cal P}\subseteq \Sigma^2$ such that the letter 
graph $G({\cal P},w)$ of the word $w=\phi(\pi)$ coincides with $G_{\pi}$. In order to define the decoder $\cal P$, 
we observe that two points $p_i$ and $p_j$ of a permutation $\pi\in {\rm Geom}(M)$ corresponds to a pair of adjacent 
vertices in $G_{\pi}$ if and only if one of them lies to the left and above the second one in the figure $\Phi_M$.
Therefore, if 
\begin{itemize}
\item $M_{k,\ell}=1$, then the points lying in the cell $C_{k,\ell}$ form an independent set in the permutation graph 
of $\pi$. Therefore, we do not include the pair $(a_{k\ell},a_{k\ell})$ in $\cal P$. 
\item $M_{k,\ell}=-1$, then the points lying in the cell $C_{k,\ell}$ form a clique in the permutation graph 
of $\pi$. Therefore, we include the pair $(a_{k\ell},a_{k\ell})$ in $\cal P$.
\item two cells $C_{k,\ell}$ and $C_{s,t}$ are independent with $k<s$ and $\ell <t$,
then no point of $C_{k,\ell}$ is adjacent to any point of $C_{s,t}$ in the permutation graph 
of $\pi$. Therefore, we include neither $(a_{k\ell},a_{st})$ nor $(a_{st},a_{k\ell})$ in $\cal P$. 
\item two cells $C_{k,\ell}$ and $C_{s,t}$ are independent with $k<s$ and $\ell > t$,
then every point of $C_{k,\ell}$ is adjacent to every point of $C_{s,t}$ in the permutation graph 
of $\pi$. Therefore, we include both pairs $(a_{k\ell},a_{st})$ and $(a_{st},a_{k\ell})$ in $\cal P$. 
\item two cells $C_{k,\ell}$ and $C_{s,t}$ share a column, i.e. $k=s$, then we look at the sign (direction) associated 
with this column and the relative position of the two cells within the column. 
\begin{itemize}
\item If $c_k=1$ (i.e. the column is oriented from left to right) and $\ell>t$ (the first of the two cells is above the second one),
then only the pair $(a_{k\ell},a_{kt})$ is included in $\cal P$.
\item If $c_k=1$ and $\ell<t$, then only the pair $(a_{kt},a_{k\ell})$ is included in $\cal P$.
\item If $c_k=-1$ (i.e. the column is oriented from right to left) and $\ell>t$ (the first of the two cells is above the second one),
then only the pair $(a_{kt},a_{k\ell})$ is included in $\cal P$.
\item If $c_k=-1$ and $\ell<t$, then only the pair $(a_{k\ell},a_{kt})$ is included in $\cal P$.
\end{itemize}
\item two cells $C_{k,\ell}$ and $C_{s,t}$ share a row, i.e. $\ell=t$, then we look at the sign (direction) associated 
with this row and the relative position of the two cells within the row. 
\begin{itemize}
\item If $r_{\ell}=1$ (i.e. the row is oriented from bottom to top) and $k<s$ (the first of the two cells is to the left of the second one),
then only the pair $(a_{s\ell},a_{k\ell})$ is included in $\cal P$.
\item If $r_{\ell}=1$ and $k>s$, then only the pair $(a_{k\ell},a_{s\ell})$ is included in $\cal P$.
\item If $r_{\ell}=-1$ (i.e. the row is oriented from top to bottom) and $k<s$,
then only the pair $(a_{k\ell},a_{s\ell})$ is included in $\cal P$.
\item If $r_{\ell}=-1$ and $k>s$, then only the pair $(a_{s\ell},a_{k\ell})$ is included in $\cal P$.
\end{itemize}
\end{itemize}
It is now a routine task to verify that $G({\cal P},w)$ coincides with $G_{\pi}$.


\subsection{From $2$-letter graphs to geometrically griddable classes of permutations}
\label{sec:back}

In this section, we prove the ``if'' part of Conjecture~\ref{con:main} for $k=2$. In other words, we prove the following result. 

\begin{theorem}
Let $X$ be a class of permutations and ${\cal G}_X$ the corresponding class of permutation graphs.
If ${\cal G}_X$ is a class of $2$-letter graphs, then $X$ is geometrically griddable.
\end{theorem}

\begin{proof}
Let $\Sigma = \{a, b\}$, and fix a decoder $\mathcal{P}$. Consider a graph $G_{\pi}\in {\cal G}_X$ and represent it by a word over $\Sigma$ with the decoder $\mathcal{P}$. 

\medskip
Assume first that $\mathcal{P}$ contains either both of $(a, b)$ and $(b, a)$, or none of them.
Then we have either all possible edges between the set of vertices of $G_{\pi}$ labelled by $a$ and the set of vertices of $G_{\pi}$  labelled by $b$ or none of them. 
In the first case,  $X$ is contained in the geometric grid class of the matrix on the left, and in the second case,  
$X$ is contained in the geometric grid class of the matrix on the right: 
$$\left(
\begin{array}{cc}
  m_a & 0 \\
  0 & m_b 
\end{array}
\right)
\mbox{ and } 
\left(
\begin{array}{cc}
   0 & m_a \\
   m_b & 0 
 \end{array}\right),
$$
where $m_a$ (resp. $m_b$) denotes either $1$ if $(a,a) \notin \mathcal{P}$ (resp. $(b,b) \notin \mathcal{P}$) or $-1$ if $(a,a) \in \mathcal{P}$ (resp. $(b,b) \in \mathcal{P}$).

\medskip
Now suppose only one of $(a, b)$ and $(b, a)$ is in $\mathcal{P}$. 
Without loss of generality assume it is $(a, b)$, since the other case is similar.

If only one of $(a, a)$ and $(b, b)$ is in $\mathcal{P}$, then $G_\pi$ is a threshold graph. In this case, $\pi$ can be placed in the figure of
$$
\left(
\begin{array}{rr}
   -1 & 1 \\
   1 & -1 
 \end{array}\right)
$$
known as $\times$-figure (see the right-hand side of  figure Figure~\ref{fig:grids}). 
Indeed, according to Proposition 5.6.1 in \cite{Waton},  a permutation can be placed in the $\times$-figure if and only if it avoids 2143, 3412, 2413 and 3142.
The first two of these permutations correspond to $2K_2$ and $C_4$, while the last two both correspond to $P_4$. 
Since a graph is threshold if and only if it is $(P_4,C_4,2K_2)$-free, we conclude that $\pi$ can be placed in the $\times$-figure, since $G_{\pi}$ is threshold.

The cases when either both or none of $(a, a)$ and $(b, b)$ belong to $\mathcal{P}$ are complement to each other. 
Therefore, we may assume without loss of generality that none of them belongs to $\mathcal{P}$. 
Then $G_\pi$ is a chain graph, and hence it is $K_3$ and $2K_2$-free. Hence $\pi$ avoids 321 and 2143. 
It is known \cite{Atkinson} that the class of permutations avoiding 321 and 2143 is the union of two classes: 
the class $A_1$ avoiding 321, 2143 and 3142, and the class $A_2$ avoiding 321, 2143 and 2413. 

A short case analysis shows that any permutation in $A_1$ can have at most one drop 
(i.e. two consecutive elements such that the first one is larger than the second one), hence it can be placed in the figure of 
$\begin{pmatrix}
1 & 1 \\
\end{pmatrix}$.
Similarly, any permutation in $A_2$ consists of two increasing subsequences such that all the elements of one of them 
are greater than every element of the other, hence it can be placed in the figure of
$\begin{pmatrix}
1 \\
1 \\
\end{pmatrix}$. 

The difference between the two classes can be illustrated as follows.   
For the class $A_1$, the word representing $G_{\pi}$ as a $2$-letter graph can be 
read at the top of the diagram representing $\pi$ (see the left diagram in Figure~\ref{fig:pp}),
while for the class $A_2$, this word can be read at the bottom of the diagram (see the right diagram in Figure~\ref{fig:pp}).
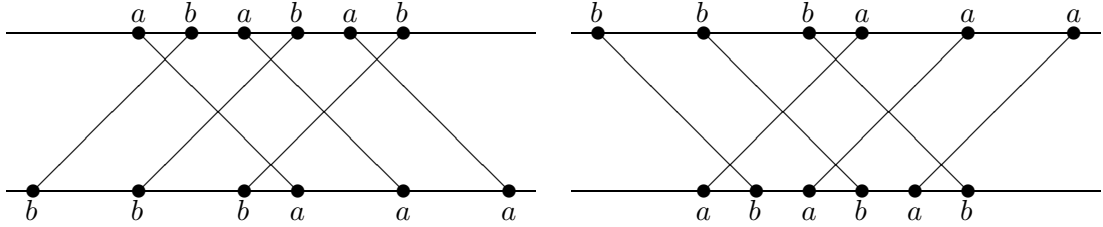
\begin{figure}[ht]
\begin{center}
\begin{picture}(210,90)
\put(40,0){\circle*{5}} 
\put(180,0){\circle*{5}}
\put(80,0){\circle*{5}} 
\put(100,0){\circle*{5}}
\put(0,0){\circle*{5}} 
\put(140,0){\circle*{5}}

\put(120,60){\circle*{5}} 
\put(40,60){\circle*{5}} 
\put(80,60){\circle*{5}} 
\put(100,60){\circle*{5}}
\put(140,60){\circle*{5}}
\put(60,60){\circle*{5}}

\put(-10,0){\line(1,0){200}} 
\put(-10,60){\line(1,0){200}} 
\put(40,0){\line(1,1){60}} 
\put(180,0){\line(-1,1){60}}
\put(80,0){\line(1,1){60}} 
\put(100,0){\line(-1,1){60}}
\put(0,0){\line(1,1){60}} 
\put(140,0){\line(-1,1){60}}
\put(57,64){$b$}
\put(97,-11){$a$} 
\put(-3,-11){$b$}
\put(97,64){$b$}
\put(137,-11){$a$} 
\put(37,-11){$b$}
\put(137,64){$b$}
\put(37,64){$a$}

\put(177,-11){$a$} 
\put(77,-11){$b$}
\put(77,64){$a$}
\put(117,64){$a$}
\end{picture}
\begin{picture}(210,90)
\put(40,0){\circle*{5}} 
\put(60,0){\circle*{5}}
\put(80,0){\circle*{5}} 
\put(100,0){\circle*{5}}
\put(120,0){\circle*{5}} 
\put(140,0){\circle*{5}}

\put(0,60){\circle*{5}} 
\put(40,60){\circle*{5}} 
\put(80,60){\circle*{5}} 
\put(100,60){\circle*{5}}
\put(140,60){\circle*{5}}
\put(180,60){\circle*{5}}

\put(-10,0){\line(1,0){200}} 
\put(-10,60){\line(1,0){200}} 
\put(40,0){\line(1,1){60}} 
\put(60,0){\line(-1,1){60}}
\put(80,0){\line(1,1){60}} 
\put(100,0){\line(-1,1){60}}
\put(120,0){\line(1,1){60}} 
\put(140,0){\line(-1,1){60}}
\put(-3,64){$b$}
\put(37,-11){$a$} 
\put(57,-11){$b$}
\put(37,64){$b$}
\put(77,-11){$a$} 
\put(97,-11){$b$}
\put(77,64){$b$}
\put(97,64){$a$}

\put(117,-11){$a$} 
\put(137,-11){$b$}
\put(137,64){$a$}
\put(177,64){$a$}
\end{picture}

\end{center}
\caption{The diagrams of two permutations $\pi$ such that $G_{\pi}$ is the graph of the word $ababab$ with ${\cal P}=\{(a,b)\}$.} 
\label{fig:pp}
\end{figure}
\end{proof}




\section{Characterization and recognition of $3$-letter graphs}
\label{sec:3letter}

To develop an efficient constructive algorithm for the recognition of $3$-letter graphs we focus on 
graphs representable over a specific decoder. For this purpose, we choose the decoder $\{(a,b),(b,c),(c,a)\}$,
because it has a nice cyclic structure simplifying some of the proofs. For other decoders, 
the recognition algorithms are similar (though  not identical), so we omit them.  

We start by presenting a decomposition theorem for 3-letter graphs with the decoder $\{(a,b),(b,c),(c,a)\}$ in Section~\ref{sec:decoder}.
This result can be viewed as a specialization of Theorem~\ref{thm:k-letter} to our particular case. It immediately leads to 
a simple polynomial-time algorithm to recognize graphs in our class, which is described  in Section~\ref{sec:nice}.
As a byproduct, the decomposition theorem of Section~\ref{sec:decoder}  also leads to the induced subgraph characterization of  
3-letter graphs with the decoder $\{(a,b),(b,c),(c,a)\}$. This result is of independent interest and is presented in Section~\ref{sec:induced}.


\subsection{Characterization of 3-letter graphs with the decoder $\{(a,b),(b,c),(c,a)\}$}
\label{sec:decoder}


To characterize $3$-letter graphs, we need a few observations about $2$-letter graphs.
Let $G=(V,E)$ be a graph and $A$ an independent set in $G$. We will say that 
a linear order $(a_1,a_2,\ldots,a_k)$ of the vertices of $A$ is
\begin{itemize}
\item[-] {\it increasing} if $i<j$ implies $N(a_i)\subseteq N(a_j)$,
\item[-] {\it decreasing} if $i<j$ implies $N(a_i)\supseteq N(a_j)$,
\item[-] {\it monotone} if it is either increasing or decreasing.
\end{itemize}

By definition, each part of a chain graph (i.e. a $2K_2$-free bipartite graph) admits a monotone ordering.
Let $G=(A\cup B,E)$ be a chain graph given together with a bipartition $V(G)=A\cup B$ of its vertices into two independent sets.
We fix an order of the parts ($A$ is first and $B$ is second),
a decreasing order for $A$, an increasing order for $B$, and call $G$ a {\it properly ordered graph}.
This notion suggests an easy way of representing a $2K_2$-free bipartite graph as a $2$-letter graph. 

Let $G=(A\cup B,E)$ be a properly ordered $2K_2$-free bipartite graph. To represent $G$ as a 2-letter graph, we fix the alphabet $\Sigma=\{a,b\}$ and decoder ${\cal P}=\{(a,b)\}$.
The word $\omega$ representing $G$ can be constructed as follows. 
To each vertex of $A$ we assign letter $a$ and to each vertex of $B$ we assign letter $b$. The $a$ letters will appear in $\omega$
in the oder in which the corresponding vertices appear in $A$ and the $b$ letters will appear in $\omega$ in the oder in which the corresponding vertices appear in $B$.
The rule defining the relative positions of $a$ vertices with respect to $b$ vertices can be described in two different ways as follows:
\begin{itemize}
\item[$R_1$] an $a$ vertex is located between the last $b$ non-neighbour (if any) and the first $b$ neighbour (if any), 
\item[$R_2$] a $b$ vertex is located between the last $a$ neighbour (if any) and the first $a$ non-neighbour (if any).  
\end{itemize}  
It is not difficult to see that both rules $R_1$ and $R_2$ define the same word and this word represents $G$.

\medskip
Now we turn to $3$-letter graphs.
Let $G=(A\cup B\cup C,E)$ be a graph whose vertex set is partitioned into three independent sets $A$, $B$, $C$
such that 
\begin{itemize}
\item[(a)] $G[A\cup B]$, $G[B\cup C]$  and $G[C\cup A]$ are $2K_2$-free bipartite graphs,
\item[(b)] there are no three vertices $a\in A$, $b\in B$, $c\in C$ inducing either a triangle $K_3$ or an anti-triangle $\overline{K}_3$.
\end{itemize}
We call any graph satisfying (a) and (b) {\it nice}. Our goal is to show that a graph $G$ is a 3-letter graph with the decoder $\{(a,b),(b,c),(c,a)\}$ if and only if it is nice. First, we prove the following lemma. 

\begin{lemma}\label{lem:nice}
Let $G=(A\cup B\cup C,E)$ be a nice graph. Then each of the independent sets $A$, $B$ and $C$ admits a linear ordering such that all three bipartite graphs
$G[A\cup B]$, $G[B\cup C]$  and $G[C\cup A]$ are properly ordered. 
\end{lemma} 

\begin{proof}
We start with a proper order of $G[A\cup B]$, in which case the order of $B$ is increasing with respect to $A$. 
Let us show that the same order of $B$ is decreasing with respect to $C$. 

Consider two vertices $b_i$ and $b_j$ of $B$ with $i<j$, i.e. $b_i$ precedes $b_j$ in the linear order of $B$ and hence $N(b_i)\cap A\subseteq N(b_j)\cap A$.
To show that the linear order of $B$ is decreasing with respect to $C$, assume the contrary: $b_j$ has a neighbour $c\in C$ non-adjacent to $b_i$. 
Without loss of generality, we may suppose that the inclusion $N(b_i)\cap A\subseteq N(b_j)\cap A$ is proper, 
since we can, if necessary, reorder all vertices with equal neighbourhoods in $A$ decreasingly with respect to their neighbourhoods in $C$, which keeps the graph $G[A\cup B]$ properly ordered.
According to this assumption, $b_j$ must have a neighbour $a\in A$ non-adjacent to $b_i$. But then either $a,b_j,c$ induce a triangle $K_3$ (if $a$ is adjacent to $c$) 
or  $a,b_i,c$ induce an anti-triangle $\overline{K}_3$ (if $a$ is not adjacent to $c$).
A contradiction in both cases shows that  the linear order of $B$ is decreasing with respect to $C$. 

Similar arguments show that the order of $A$ which is decreasing with respect to $B$ is increasing with respect to $C$. 
Now we fix a linear order of $C$ which is increasing  with respect to $B$ and conclude, as before, that it is decreasing with respect to $A$. In this way, 
we obtain a proper order for all three graphs $G[A\cup B]$, $G[B\cup C]$  and $G[C\cup A]$ (notice, in the last graph $C$ is the first part and $A$ is the second).
\end{proof}

\begin{theorem}\label{thm:nice}
A graph $G$ is a 3-letter graph with the decoder $\{(a,b),(b,c),(c,a)\}$ if and only if it is nice.
\end{theorem}

\begin{proof}
If $G$ is a 3-letter graph with the decoder $\{(a,b),(b,c),(c,a)\}$, then obviously $V_a$ (the set of vertices labelled by $a$), $V_b$ and $V_c$ 
are independent sets and condition (a) of the definition of nice graphs  is valid for $G$. To show that (b) is valid, assume $G$ contains a triangle induced by letters $a,b,c$. Then $b$ must appear
after $a$ in the word representing $G$, and $c$ must appear after $b$. But then $c$ appears after $a$, in which case $a$ is not adjacent to $c$, a contradiction. 
Similarly, an anti-triangle $a,b,c$ is not possible and hence $G$ is nice. 

Suppose now that $G=(A\cup B\cup C,E)$ is nice. According to Lemma~\ref{lem:nice}, we may assume that $A$, $B$ and $C$ are ordered in such a way that each of the three bipartite graphs
$G[A\cup B]$, $G[B\cup C]$  and $G[C\cup A]$ is properly ordered. 

We start by representing the graph $G[A\cup B]$ by a word $\omega$ with two letters $a,b$ according to rules $R_1$ or $R_2$. To complete the construction, we need to place the $c$ vertices 
\begin{itemize}
\item[-] among the $a$ vertices according to rule $R_1$, i.e. every $c$ vertex must be located between  the last $a$ non-neighbour $a_{lnn}$ (if any) and the first $a$ neighbour $a_{fn}$ (if any),
\item[-] among the $b$ vertices according to rule $R_2$, i.e. every $c$ vertex must be located between  the last $b$ neighbour $b_{ln}$ (if any) and the first $b$ non-neighbour $b_{fnn}$ (if any).
\end{itemize}
This is always possible, unless
\begin{itemize}
\item[-] either $a_{fn}$ precedes $b_{ln}$ in $\omega$, in which case $a_{fn}$ is adjacent to $b_{ln}$ and hence $a_{fn},b_{ln},c$ induce a triangle $K_3$,
\item[-] or $b_{fnn}$  precedes $a_{lnn}$ in $\omega$,  in which case $b_{fnn}$ is not adjacent to $a_{lnn}$  and hence $a_{lnn},b_{fnn},c$ induce an anti-triangle $\overline{K}_3$.
\end{itemize}
A contradiction in both case shows that $\omega$ can be extended to a word representing $G$. 
\end{proof}


\subsection{Recognition of 3-letter graphs with the decoder $\{(a,b),(b,c),(c,a)\}$}
\label{sec:nice}


In this section, we show how we can determine whether a graph $G$ can be represented as a 3-letter graph with the cyclic decoder $\{(a,b),(b,c),(c,a)\}$. In order to do that, we will assume that $G$ did indeed have such a representation $\omega$, and derive various properties of $\omega$.

If $G$ has a twin $v$ for a vertex $u$ (i.e. $N(v)=N(u)$), then any word representing $G-v$ can be extended to a word representing $G$ by 
assigning to $v$ the same letter as to $u$ and placing $v$ next to $u$.  This observation shows that we may assume without loss of generality that 
\begin{itemize}
\item $G$ is twin-free.
\end{itemize}
Due to the cyclic symmetry of the decoder, we may also assume without loss of generality that 
\begin{itemize}
\item the last letter of $\omega$ is $c$. 
\end{itemize}
Then 
\begin{itemize}
\item the first letter is not $a$, since otherwise the first and the last vertices are twins.
\end{itemize}

Assume that the first letter of $\omega$ is $b$. Then according to the decoder
\begin{itemize}
\item[(b1)] no vertex  between the first $b$ and the last $c$ is adjacent to both of them,  
\item[(b2)]  every vertex non-adjacent to the first $b$ and non-adjacent to the last $c$ must be labelled by $a$,
\item[(b3)]  every vertex non-adjacent to the first $b$ and adjacent to the last $c$ must be labelled by $b$,
\item[(b4)] every vertex adjacent to the first $b$ and non-adjacent to the last $c$ must be labelled by $c$. 
\end{itemize}
Therefore, in order to determine whether $G$ can be represented by a word starting with $b$ and ending with $c$, we 
\begin{itemize}
\item[-] inspect every pair of adjacent vertices and assign letter $b$ to one of them and letter $c$ to the other, 
\item[-] check whether the set of vertices adjacent to both of them is empty and split the remaining vertices of the graph into three subsets $A, B,C$ according to (b2), (b3) and (b4), respectively, 
\item[-] verify whether the partition obtained in this way satisfies the definition of nice graphs (conditions (a) and (b)).
\end{itemize}
\medskip
Finally, we determine whether $G$ can be represented by a word starting with $c$. Then according to the decoder
\begin{itemize}
\item[(c1)] no vertex between the first $c$ and the last $c$ is adjacent to both of them,  
\item[(c2)] every vertex adjacent to the first $c$ and non-adjacent to the last $c$ must be labelled by $a$, 
\item[(c3)]  every vertex non-adjacent to the first $c$ and adjacent to the last $c$ must be labelled by $b$,
\item[(c4)]  every vertex non-adjacent to the first $c$ and non-adjacent to the last $c$ must be labelled by $c$.
\end{itemize}
Therefore, in order to determine whether $G$ can be represented by a word starting with $c$ and ending with $c$, we 
\begin{itemize}
\item[-] inspect every pair of non-adjacent vertices and assign letter $c$ to both of them, 
\item[-] check whether the set of vertices adjacent to both of them is empty and split the remaining vertices of the graph into three subsets $A, B,C$ according to (c2), (c3) and (c4), respectively,
\item[-] verify whether the partition obtained in this way satisfies the definition of nice graphs (conditions (a) and (b)).
\end{itemize}
\medskip
From the above discussion we derive the following conclusion.

\begin{theorem}
The 3-letter graphs with the decoder $\{(a,b),(b,c),(c,a)\}$ can be recognized in polynomial time. 
\end{theorem}


\subsection{Minimal forbidden induced subgraphs for 3-letter graphs with the decoder $\{(a,b),(b,c),(c,a)\}$}
\label{sec:induced}

To determine the list of minimal forbidden induced subgraphs for our class,
we will rely on our earlier characterization of graphs in this class as ``nice'' (Theorem~\ref{thm:nice}).
We start with a preparatory result. 

\begin{lemma}\label{joininglemma}
Let $G$ be a graph and let $H_1$ and $H_2$ be nice subgraphs of $G$ with disjoint vertex sets $V(H_1) = A_1\cup B_1\cup C_1$ and $V(H_2) = A_2\cup B_2\cup C_2$. 
If the subgraphs induced 
\begin{itemize}
\item by $A_1\cup B_2$, $B_1\cup C_2$ and $C_1\cup A_2$ are complete bipartite, 
\item by $A_1\cup A_2$, $A_1\cup C_2$, $B_1\cup A_2$, $B_1\cup B_2$, $C_1\cup B_2$ and $C_1\cup C_2$ are edgeless,
\end{itemize}
then the subgraph induced by $V(H_1) \cup V(H_2) = (A_1\cup A_2)\cup (B_1\cup B_2)\cup (C_1\cup C_2)$  is nice.
\end{lemma}

\begin{proof}
By assumption, $A_1\cup A_2$ and $B_1\cup B_2$ are independent sets. Let us show that these two sets induce a chain graph. 
First, it is not difficult to see that $A_1\cup (B_1\cup B_2)$ induces a chain graph, because $G[A_1\cup B_1]$ is a chain graph and 
$G[A_1\cup B_2]$ is complete bipartite. Similar arguments show that $A_2\cup (B_1\cup B_2)$, $(A_1\cup A_2)\cup B_1$ and $(A_1\cup A_2)\cup B_2$ all induce chain graphs. 
Therefore, if the subgraph of $G$ induced by $A_1\cup A_2$ and $B_1\cup B_2$ contains an induced $2K_2$,
then this $2K_2$ contains exactly one vertex in  each of the four sets, which is impossible. 
This contradiction shows that the subgraph of $G$ induced by $A_1 \cup A_2$ and $B_1 \cup B_2$ is a chain graph.

By symmetry, $(B_1\cup B_2)\cup (C_1\cup C_2)$ and $(C_1\cup C_2)\cup (A_1\cup A_2)$ also induce chain graphs. 
It remains to show that no 3 vertices $a \in A_1\cup A_2$, $b \in B_1\cup B_2$, $c \in C_1\cup C_2$ induce a triangle or an anti-triangle. 
Since $H_1$ and $H_2$ are nice, we may assume without loss of generality that two of the vertices belong to $H_1$ and one to $H_2$.
Also, due to the symmetry of the decoder, we may assume that $a \in A_1$, $b \in B_1$, and $c \in C_2$. Then $a,b,c$ induce neither a triangle (since $a$ is not adjacent to $c$)
nor an anti-triangle (since $b$ is adjacent to $c$).
\end{proof}

We are now ready to prove the characterization in terms of minimal forbidden induced subgraphs.

\begin{theorem}\label{mfischar}
A graph $G$ is a 3-letter graph with decoder $\{(a, b), (b, c), (c, a)\}$ if and only if it is $(K_3, 2K_2 + K_1, C_5 + K_1, C_6)$-free.
\end{theorem}

\begin{proof}
For the ``only if'' direction, it is straightforward to check that none of the four graphs in our list is nice and that they are minimal with that property. 
For the  ``if'' direction, we split the analysis into two cases.

\medskip
Assume first that $G$ is $2K_2$-free. If, in addition, it is $C_5$-free, then $G$ is $2K_2$-free bipartite, i.e. a chain graph 
(since it has no $2K_2$, $K_3$, $C_5$, and the absence of $2K_2$ forbids longer odd cycles), 
hence it is nice, with one of the 3 sets being empty. So suppose $G$ has an induced $C_5$. Label its vertices clockwise by $v_1, \ldots, v_5$ 
(whenever indices are added in this proof, the addition will be modulo 5). 
Any vertex of $G$ not in the $C_5$
\begin{itemize}
\item has to be adjacent to at least one vertex in the $C_5$, since otherwise an induced  $C_5 + K_1$ arises, 
\item cannot have a single neighbour in the $C_5$, since otherwise an induced $2K_2$ can be easily found, 
\item cannot be adjacent to 3 or more vertices or to 2 consecutive vertices in the $C_5$, since $G$ is $K_3$-free.
\end{itemize}
Hence the vertices of $G$ can be partitioned into 5 sets $V_1, \ldots, V_5$ such that the vertices in $V_i$ are adjacent to only $v_{i-1}$ and $v_{i+1}$ in the $C_5$ (note $v_i \in V_i$ for $i = 1,\ldots, 5$). 
Each $V_i$ is an independent set (since they share a common neighbour, and triangles are forbidden), and adjacency between them is easy to determine: 
\begin{itemize}
\item if $u_i \in V_i$, $u_{i+1} \in V_{i+1}$, then $u_i$ and $u_{i+1}$ are adjacent, since otherwise $u_i$, $v_{i-1}, u_{i+1}, v_{i+2}$ induce a $2K_2$, 
\item if $u_i \in V_i$, $u_{i+2} \in V_{i+2}$, then $u_i$ and $u_{i+2}$ are non-adjacent, since otherwise $u_i,v_{i+1},u_{i+2}$ induce a triangle.
\end{itemize} 
This determines all adjacencies in $G$, and it is easy to check that $G$ is nice, e.g. with partition $(V_1\cup V_4), (V_2\cup V_5), V_3$.

\medskip
Now we turn to the case when $G$ contains an induced $2K_2$. 
We denote one of the edges of the $2K_2$ by $uw$ and partition the vertices of $G$ into three subsets as follows 
(observe that there are no vertices adjacent to both $u$ and $w$, since triangles are forbidden):
\begin{itemize}
\item[$U$] is the set of vertices adjacent to $w$ ($u$ belongs to $U$). Since triangles are forbidden, $U$ is an independent set.
\item[$W$] is the set of vertices adjacent to $u$ ($w$ belongs to $W$). Since triangles are forbidden, $W$ is an independent set.
\item[$X$] is the set of vertices adjacent neither to $u$ nor to $w$. The subgraph induced by $X$ must be $K_2 + K_1$-free, since otherwise an induced copy of $2K_2 + K_1$ arise.
It is not difficult to see that $(K_2 + K_1, K_3)$-free graphs are precisely complete bipartite graphs. Therefore, the vertices  of $X$ can be split into two independent sets
with all possible edges between them. We call these independent sets  $C_1$ and $A_2$ (this notation is chosen for consistency with Lemma~\ref{joininglemma})
and observe that each of them is non-empty, because $X$ contains the other edge of the $K_2$. 
\end{itemize}
Since $G$ is $K_3$-free, no vertex of $G$ can have neighbours in both $C_1$ and $A_2$. Thus $W$ can be partitioned into three subsets as follows:
\begin{itemize}
\item[$A_1$] is the vertices of $W$ that do have neighbours in $C_1$ (and hence have no neighbours in $A_2$),
\item[$C_2$] is the vertices of $W$ that do have neighbours in $A_2$ (and hence have no neighbours in $C_1$),
\item[$W'$] is the set of remaining vertices of $W$, i.e. those that have neighbours neither in $C_1$ nor in $A_2$.
\end{itemize}
We partition $U$ into three subsets in a similar way:
\begin{itemize}
\item[$B_1$] is the vertices of $U$ that do have neighbours in $C_1$ (and hence have no neighbours in $A_2$),
\item[$B_2$] is the vertices of $U$ that do have neighbours in $A_2$ (and hence have no neighbours in $C_1$),
\item[$U'$] is the set of remaining vertices of $U$, i.e. those that have neighbours neither in $C_1$ nor in $A_2$.
\end{itemize}
We note that 
\begin{itemize}
\item {\it Every vertex of $A_1$ is adjacent to every vertex of $B_2$}. Indeed, if  $a_1 \in A_1$ is not adjacent to  $b_2 \in B_2$, 
then $u,w,a_1,b_2$ together with a neighbour of $a_1$ in $C_1$ and a neighbour of $b_2$ in $A_2$ induce a  $C_6$. 
\item {\it Every vertex of $B_1$ is adjacent to every vertex of $C_2$} by similar arguments.
\item {\it Every vertex of $U'$ is adjacent to every vertex of $W'$}. Indeed, if $u'\in U'$ is not adjacent to $w'\in W'$, 
then $u',w',w$ together with any two vertices $c_1\in C_1$ and $a_2\in A_2$ induce a $2K_2+K_1$.
\item {\it Every vertex of $W'$ is adjacent to either every vertex in $B_1$ or to every vertex in $B_2$}. Indeed, if a vertex $w'\in W'$ has 
a non-neighbour $b_1\in B_1$ and a non-neighbour $b_2\in B_2$, then  $w',w,b_1,b_2$ together with a neighbour of $b_1$ in $C_1$ and a neighbour of $b_2$ in $A_2$ induce a  $C_5+K_1$. 
\item {\it Every vertex of $U'$ is adjacent to every vertex either in $A_1$ or in $C_2$} by similar arguments.
\end{itemize}
The above sequence of claims shows that we can move the vertices from $W'$ to either $A_1$ or $C_2$ and those from $U'$ to either $B_1$ or $B_2$
in such a way that the two subgraphs $G[A_1\cup B_2]$ and $G[B_1\cup C_2]$ are complete bipartite.

To sum up, we have partitioned $G$ into independent sets $A_1, A_2, B_1, B_2, C_1$, and $C_2$, 
such that $G[A_1\cup B_2]$, $G[B_1\cup C_2]$ and $G[C_1\cup A_2]$ are complete bipartite, 
while $G[A_1\cup A_2]$, $G[A_1\cup C_2]$, $G[B_1\cup A_2]$, $G[B_1\cup B_2]$, $G[C_1\cup B_2]$ and $G[C_1\cup C_2]$ are edgeless. 
To apply Lemma~\ref{joininglemma} it remains to show that $G[A_1\cup B_1\cup C_1]$ and $G[A_2\cup B_2\cup C_2]$ are nice.

Because of the $2K_2+K_1$-freeness, the subgraph induced by the set of non-neighbours of any vertex is $2K_2$-free. 
Therefore, each of $G[A_1\cup B_1]$, $G[B_1\cup C_1]$ and $G[C_1\cup A_1]$ is $2K_2$-free, since they are induced by non-neighbours of $a_2, u, w$, respectively 
(where $a_2$ is an arbitrary  vertex in $A_2$, which exists because $A_2$ is not empty). 
We do not need to worry about triangles, since they are forbidden anyway. Finally, if there was an anti-triangle induced by $a_1\in A_1, b_1\in B_1, c_1\in C_1$, 
then together with $u$ and any vertex $a_2 \in A_2$ they would induce a  $2K_2 + K_1$. 
This shows that $G[A_1\cup B_1\cup C_1]$ is nice. The other subgraph is treated analogously.
Therefore, by Lemma~\ref{joininglemma} $G$ is nice.
\end{proof}

\section{Concluding remarks and open problems}
\label{sec:conclusion}
The notion of letter graphs studied in this paper is relatively new. To better understand its relation to the existing notions, 
we discuss in Section~\ref{sec:para} the position of  graph lettericity in the hierarchy of other graph parameters.
Then we conclude the paper in Section~\ref{sec:open} with a number of open problems.

\subsection{Graph lettericity in the hierarchy of graph parameters}
\label{sec:para}

In this section  we show that lettericity is squeezed between neighbourhood diversity and linear clique-width 
in the sense that bounded neighbourhood diversity implies bounded lettericity, which in turn implies bounded linear clique-width. 

\medskip
We start with the neighbourhood diversity. This parameter was introduced in \cite{Lampis} to study parameterized complexity of algorithmic graph problems
and can be defined as follows.

\begin{definition}
Two vertices $x$ and $y$ are said to be {\it similar} if there is no third vertex $z$ distinguishing them
(i.e. if there is no third vertex $z$ adjacent to exactly one of $x$ and $y$). Clearly, the similarity is an equivalence relation. 
The number of similarity classes is the {\it neighbourhood diversity} of $G$.
\end{definition}

\begin{theorem}
If the neighbourhood diversity of $G$ is $k$, then the lettericity of $G$ is at most $k$. 
\end{theorem}

\begin{proof}
It is not difficult to see that every similarity class in a graph is either a clique or an independent set
and between any two similarity classes we have either all possible edges or none of them. Therefore,  
if the neighbourhood diversity of $G$ is $k$, then we need at most $k$ letters to represent $G$ as a $k$-letter graph (one letter per similarity class). 
If a similarity class corresponding to letter $a$ is a clique, we include the pair $(a,a)$ in the decoder, otherwise we do not. 
Also, if two similarity classes corresponding to letters $a$ and $b$ are complete to each other (all possible edges between them),
we include both pairs $(a,b)$ and $(b,a)$ in the decoder, otherwise we include neither of them. With the decoder constructed in this way, 
{\it any} word over the alphabet of $k$ letters represents $G$.  
\end{proof}

\medskip
Now we turn to linear clique-width. This is a restriction of a more general parameter clique-width. The {\em clique-width}
of a graph $G$
is the minimum number of labels needed to construct $G$ using
the following four operations:
\begin{itemize}
\item[(i)] Creation of a new vertex $v$ with label $i$ (denoted by $i(v)$).
\item[(ii)] Disjoint union of two labeled graphs $G$ and $H$
(denoted by $G\oplus H$).
\item[(iii)] Joining by an edge each vertex with label $i$ to each vertex with label $j$
($i\not= j$, denoted by $\eta_{i,j}$).
\item[(iv)] Renaming label $i$ to $j$
(denoted by $\rho_{i\to j}$).

\end{itemize}
Every graph can be defined by an algebraic expression using the four operations above. 
This expression is called a $k$-expression if it uses $k$ different labels.
For instance, the cycle $C_5$ on vertices $a,b,c,d,e$ 
(listed along the cycle) can be defined by the following 4-expression:
$$
\eta_{4,1}(\eta_{4,3}(4(e)\oplus\rho_{4\to 3}(\rho_{3\to 2}(\eta_{4,3}(4(d)\oplus\eta_{3,2}(3(c)\oplus\eta_{2,1}(2(b)\oplus 1(a)))))))).
$$

Alternatively, any algebraic expression defining $G$ can be represented as a rooted tree, 
whose leaves correspond to the operations of vertex creation, the internal nodes correspond 
to the $\oplus$-operations, and the root is associated with $G$. The operations $\eta$ and 
$\rho$ are assigned to the respective edges of the tree. Figure~\ref{fig:tree} shows the tree 
representing the above expression defining a $C_5$. 

\bigskip
\begin{figure}[ht]
\begin{center}
\setlength{\unitlength}{0.32mm}
\begin{picture}(370,50)
\put(15,50){\circle{20}}
\put(90,50){\circle{20}}
\put(165,50){\circle{20}}
\put(240,50){\circle{20}}
\put(315,50){\circle{20}}
\put(355,50){\circle{20}}
\put(90,10){\circle{20}}
\put(165,10){\circle{20}}
\put(240,10){\circle{20}}
\put(315,10){\circle{20}}
\put(25,50){\line(1,0){55}}
\put(100,50){\line(1,0){55}}
\put(175,50){\line(1,0){55}}
\put(250,50){\line(1,0){55}}
\put(325,50){\line(1,0){20}}
\put(90,40){\line(0,-1){20}}
\put(165,40){\line(0,-1){20}}
\put(240,40){\line(0,-1){20}}
\put(315,40){\line(0,-1){20}}
\put(85,47){+}
\put(160,47){+}
\put(235,47){+}
\put(311,47){+}
\put(8,47){$C_5$}
\put(83,10){$_{4(e)}$}
\put(158,10){$_{4(d)}$}
\put(233,10){$_{3(c)}$}
\put(308,10){$_{2(b)}$}
\put(348,50){$_{1(a)}$}
\put(101,55){$_{\rho_{4\to 3}\rho_{3\to 2}\eta_{4,3}}$}
\put(36,55){$_{\eta_{4,1}\eta_{4,3}}$}
\put(196,55){$_{\eta_{3,2}}$}
\put(271,55){$_{\eta_{2,1}}$}
\end{picture}
\end{center}
\caption{The tree representing the expression defining a $C_5$}
\label{fig:tree}
\end{figure}
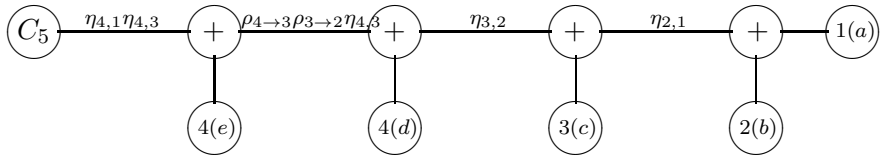

Let us observe that the tree in Figure~\ref{fig:tree} has a special form known as a {\it caterpillar
tree} (that is, a tree that becomes a path 
after the removal of vertices of degree 1). The minimum number of labels needed to construct a graph $G$ by means of caterpillar trees is 
called the {\it linear clique-width} of $G$ and is denoted ${\rm lcwd}(G)$. Clearly, ${\rm lcwd}(G) \ge {\rm cwd}(G)$

\begin{theorem}
${\rm lcwd}(G)\le \ell(G) +1$.
\end{theorem}

\begin{proof}
Let $w=w_1,w_2,\ldots,w_n$ be a word defining a graph $G$ with vertex set $\{v_1,\ldots,v_n\}$ over an alphabet $\Sigma=\{a_1,\ldots,a_k\}$ of $k\le\ell(G)$ letters with a decoder $\cal P$. 
To construct a linear clique-width expression for $G$ we will use $k+1$ labels  $a_0,a_1,\ldots,a_k$ as follows.
Assume the subgraph of $G$ induced by the first $i-1$ vertices has been constructed in such a way that in the end of the construction the label of $v_j$ is $w_j$ for each $j=1,\ldots,i-1$. 
A new vertex $v_i$ is created with label $a_0$. Then for each $j=1,\ldots,i-1$ we connect $w_j$ to $a_0$ whenever $(w_j,w_i)\in \cal P$ and rename $a_0$ to $w_i$. 
It is not difficult to see that this procedure creates the graph $G$. 
\end{proof}

\subsection{Open problems}
\label{sec:open}

In this paper, we revealed a relationship between  letter graphs and geometrically griddable permutations. 
We also gave a partial description of this relationship. 
However, describing the relationship in its whole generality remains an open problem. 

One more open problem is the development of constructive algorithms for the recognition of $k$-letter graphs.
For $k=2$, a solution to this problem follows from the results in \cite{letter-graphs}, where the author 
gave a complete characterization of $2$-letter graphs for each possible decoder. 
This naturally leads to a quadratic algorithm to recognize the $2$-letter graphs.
In the present paper, we studied $3$-letter graphs  representable over a specific decoder
and characterized this class both structurally and in terms of minimal forbidden induced subgraphs. 
As a result, we obtained a polynomial-time algorithm for the recognition of graphs in this class. 
Similar ideas can be used for the recognition of $3$-letter graphs with other decoders. 
However, a more challenging task is the development of constructive algorithms independent of the decoders. 

It is also remains open whether the lettericity of a graph can be computed in polynomial time. 
Note that for the related parameter linear clique-width this problem is NP-complete \cite{cw}. 



\begin{thebibliography}{99}
\bibitem{grid-classes}
M.H. Albert, M.D. Atkinson, M. Bouvel, N. Ruskuc, V. Vatter, 
{\it Geometric grid classes of permutations}, 
Trans. Amer. Math. Soc. 365 (2013), 5859--5881.

\bibitem{iwoca}
B. Alecu, V. Lozin, V. Zamaraev, D. de Werra,
{\it Letter graphs and geometric grid classes of permutations: characterization and recognition}.
International Workshop on Combinatorial Algorithms, Springer, Cham, 2017.


\bibitem{Atkinson}
M.D. Atkinson,  
{\it Restricted permutations}, 
Discrete Mathematics, 195 (1999), 27--38.


\bibitem{threshold-graphs}
V. Chv\'atal, P.L. Hammer,
{\it Aggregation of inequalities in integer programming}. 
Ann. of Discrete Math., 1 (1977) 145--162.

\bibitem{Damaschke}
P. Damaschke,
{\it Induced subgraphs and well-quasi-ordering}. 
J. Graph Theory 14 (1990), no. 4, 427--435.

\bibitem{Diekert}
V. Diekert,
{\it Combinatorics on traces}.
Lecture Notes in Computer Science, vol. 454, Springer-Verlag, Berlin, 1990. 

\bibitem{cw}
M.R. Fellows, F.A. Rosamond, U. Rotics, S. Szeider, 
{\it Clique-width is NP-complete}. 
SIAM J. Discrete Math. 23 (2009), no. 2, 909--939.

\bibitem{universal-threshold}
P.L. Hammer, A.K. Kelmans, 
{\it On universal threshold graphs}. 
Combin. Probab. Comput.  3  (1994),  no. 3, 327--344.

\bibitem{two}
N. Korpelainen and V. Lozin, 
{\it Two forbidden induced subgraphs and well-quasi-ordering}. 
Discrete Mathematics, 311 (2011), no. 16, 1813--1822.

\bibitem{Lampis}
M. Lampis, 
{\it Algorithmic meta-theorems for restrictions of treewidth}.
Algorithmica, 64 (2012), no. 1, 19--37. 

\bibitem{chain-uni}
V. Lozin, G. Rudolf,
{\it Minimal universal bipartite graphs}.  
Ars Combin.  84  (2007), 345--356. 

\bibitem{threshold}
N. V. R. Mahadev and U.N. Peled,  
{\it Threshold graphs and related topics}. 
Annals of Discrete Mathematics, 56. North-Holland Publishing Co., Amsterdam, 1995. xiv+543 pp.

\bibitem{ordered}
J. Ne\v set\v ril, 
{\it On ordered graphs and graph orderings}. 
Discrete Appl. Math. 51 (1994), no. 1-2, 113--116.

\bibitem{letter-graphs}
M. Petkov\v sek, 
{\it Letter graphs and well-quasi-order by induced subgraphs},
Discrete Mathematics, 244 (2002) 375--388.

\bibitem{Waton}
S.D. Waton, 
{\it On permutation classes defined by token passing networks, gridding matrices and pictures: three flavours of involvement}. 
Ph.D. thesis. University of St Andrews, St Andrews, 2007. Available at \\
{\small \verb|https://pdfs.semanticscholar.org/f45e/4ff6160425fbb0068f291be7a7d36ef64651.pdf|}

\end{thebibliography}
\end{document}